\documentclass[12pt,a4paper,oneside]{amsart}
\usepackage{amsfonts, amsmath, amssymb, amsthm}
\usepackage[colorlinks=true,citecolor=blue]{hyperref}
\usepackage[margin=1.4in]{geometry}
\usepackage{txfonts}
\usepackage[T1]{fontenc}
\usepackage{bbm}
\usepackage{mathtools}

\usepackage{graphicx}
\usepackage{enumerate}
\usepackage{verbatim}
\usepackage{tikz}
\usepackage{algorithmic}

\usetikzlibrary{decorations,decorations.pathreplacing}

\begingroup
    \makeatletter
    \@for\theoremstyle:=definition,remark,plain\do{%
        \expandafter\g@addto@macro\csname th@\theoremstyle\endcsname{%
            \addtolength\thm@preskip\parskip
            }%
        }
\endgroup

\newtheorem*{theorem*}{Theorem}

\newtheorem*{lemma*}{Lemma}

\newtheorem*{claim*}{Claim}

\theoremstyle{definition}

\newtheorem*{remark*}{Remark}

\newcommand{\cal}{\mathcal}
\newcommand{\conv}{\text{conv}}

\begin{document} 

\title[A colorful Goodman--Pollack--Wenger theorem]{A colorful Goodman--Pollack--Wenger theorem} 

\author{Andreas F. Holmsen}

\date{\today}

 \address{Andreas F. Holmsen, 
 Department of Mathematical Sciences, 
 KAIST, 
 Daejeon, South Korea.  \hfill \hfill }
 \email{andreash@kaist.edu}

\begin{abstract} 
Hadwiger's transversal theorem gives necessary and sufficient conditions for the existence of a line transversal to a family of pairwise disjoint convex sets in the plane. These conditions were subsequently generalized to hyperplane transversals to general families of convex sets in $\mathbb{R}^d$ by Goodman, Pollack, and Wenger. Here we show a colorful genealization of their theorem which confirms a conjecture of Arocha, Bracho, and Montejano. The proof is topological and uses the Borsuk--Ulam theorem.
\end{abstract}


\maketitle 

\section{Introduction}

One of the classical results of geometric transversal theory is Hadwiger's transversal theorem \cite{had57} which asserts the following: 

\begin{quote}
{\em A finite family of pairwise disjoint convex sets in the plane has a line transversal if and only if the members of the family can be linearly ordered such that any three members have a line transversal consistent with ordering.}
\end{quote}

Here a line transversal means a line that intersects every member of the family, and for three members $A$, $B$, $C$ to a have a line transversal consistent with the ordering means that if their relative order is $A\prec B \prec C$, then we can choose points $a\in A$, $b\in B$, $c\in C$ such that $b$ is contained in the segment spanned by $a$ and $c$.

\medskip

Hadwiger's transversal theorem has been generalized in a number of ways, eventually resulting in necessary and sufficient conditions for the existence of a {\em hyperplane transversal} to a finite family of convex sets in $\mathbb{R}^d$, meaning an affine hyperplane in $\mathbb{R}^d$ that intersects every member of the family. 

The first siginificant step towards these necessary and sufficient conditions was made by Goodman and Pollack \cite{eli}, who showed that when going to higher dimensions, the linear ordering in Hadwiger's theorem can be replaced by the {\em order type} of a set of points in $\mathbb{R}^{d-1}$. 
The next crucial step was made by Wenger \cite{wen90}, who discovered that the disjointness condition in Hadwiger's theorem is superfluous.  
The final step to arbitrary dimension, without any type of disjointness or separation condititions, was made by Pollack and Wenger \cite{wenger}. 
(We also note that Anderson and Wenger \cite{AW} subsequently showed that the order type can in fact be replaced by an {\em acyclic oriented matroid} of rank $d$.)
For more information on geometric transversals and Helly-type theorems, see e.g. \cite{eli-survey, hol-sur}. 

\medskip

The purpose of this note is to prove a so-called {\em colorful generalization} of the Goodman--Pollack--Wenger theorem. This was originally conjectured by Arocha, Bracho, and Montejano \cite{col-had}, who proved the planar case, i.e. a colorful version of Hadwiger's theorem. See also \cite{edgar} for some partial results.

There are a number of classical theorems in combinatorial convexity that admit colorful generalizations, originating with the colorful Carath{\'e}odory theorem due to B{\'a}r{\'a}ny \cite{colhel} (and the colorful Helly theorem due to Lov{\'a}sz). These generalizations are of significant interest because they often have important applications that can not be obtained from their classical ``black-and-white'' versions. See \cite[chapters 8--10]{dgbook} and \cite{bar-book} for further discussions, examples, and references. 

Here we consider the even more general notion of {\em matroid colorings}. In the setting of combinatorial convexity, such colorings were first introduced by Kalai and Meshulam \cite{KM05} who proved a topological Helly theorem for matroid colorings. The special case of a partition matroid gives the usual colorful version.

\subsubsection*{Notation and terminology} We write $\conv\hspace{.3ex} X$ to denote the convex hull of a subset $X\subset \mathbb{R}^d$, and for a family $F$ of subsets of $\mathbb{R}^d$, we write $\conv \hspace{.3ex} F$ to denote the convex hull of the union of the members in $F$. Matroids will be specified by their rank function, and all matroids are assumed to be loopless. For an abstract simplicial complex $K$ with vertex set $V$, and a subset $W\subset V$, we write $K[W]$ to denote the induced subcomplex on the vertex set $W$. The proof of our main result uses some topology, in particular the notions of $\mathbb{Z}_2$-complexes and $\mathbb{Z}_2$-maps, and we assume the reader is familiar with the terminology from Matou{\v s}ek's book \cite{bubook}.

\subsubsection*{Main result} We are ready to state our generalization of the Goodman--Pollack--Wenger theorem for hyperplane transversals.

\begin{theorem*}\label{thm:matroid}
Let $F$ be a finite family of convex sets in $\mathbb{R}^d$ and let $\mu: 2^F \to \mathbb{N}_0$ be the rank function of a matroid. Suppose, for some $0\leq k<d$, there is a function $\varphi: F \to \mathbb{R}^k$ such that 
\[\text{\em conv}\: G_1 \cap \text{\em \conv}\: G_2 = \emptyset \implies \text{\em conv}\: \varphi(G_1) \cap \text{\em \conv}\: \varphi(G_2) = \emptyset \tag{$\ast$}\]
whenever $G_1\cup G_2\subset F$ satisfies $\mu(G_1\cup G_2) = |G_1|+|G_2|$.
Then there exists a subfamily $G\subset F$, with $\mu(F\setminus G) \leq k+1$, which has a hyperplane transversal. 
\end{theorem*}

\subsubsection*{Remark} One can interpret the rank function as a measure of the size of the subfamilies of $F$. So, roughly speaking, the theorem states that if all independent sets of the matroid satisfy $(\ast)$, then there is a relatively large subfamily $G$ that has a hyperplane transversal. (Note that if $\mu(F)\leq k+1$, then the theorem is trivially true with $G=\emptyset$.) 

We illustrate this interpretation by the following special case of the theorem. Suppose $F = F_1 \cup \cdots \cup F_{k+2}$ is a partition into nonempty parts, and define $\mu$ to be  the rank function of the corresponding partition matroid. 
The bases of the matroid are the {\em colorful $(k+2)$-tuples} of $F$, that is, the $(k+2)$-element subsets of $F$ which have exactly one member from each of the parts $F_i$. The condition $(\ast)$ asserts that any colorful $(k+2)$-tuple has Radon partitions that are consistent with the Radon partitions of its image under $\varphi$. Since this image lies in $\mathbb{R}^k$, there is at least one non-trivial Radon partition, and it follows that the colorful $(k+2)$-tuple admits a $k$-dimensional transversal which is consistent with the order type of its image under $\varphi$. Finally, note that the condition $\mu(F\setminus G)\leq k+1$ implies that $F_i\subset G$ for some $1\leq i \leq k+2$, and so the conclusion of the theorem asserts that one of the color classes $F_i$ has a hyperplane transversal. (We also note that if we set $F = F_1 = \cdots = F_{k+2}$, i.e. we take a disjoint union of $k+2$ copies of $F$,  then we get the original Goodman--Pollack--Wenger theorem.)

\section{Proof of the Theorem} Our goal is to show that the existence of a counterexample leads to a contradiction of the Borsuk--Ulam theorem.

We may assume that the members of $F$ are convex polytopes.
Indeed, if $B_1$, $\dots$, $B_\ell$, $C_1$, $\dots$, $C_m$ are convex sets such that \[\conv\{B_1, \cdots , B_\ell\} \cap \conv\{C_1,  \cdots , C_m\} \neq \emptyset,\]
then we can find points 
$b_1, \dots, b_\ell$, $c_1, \dots, c_m$, with $b_i\in B_i$ and $c_j\in C_j$, such that 
\[\conv \{b_1, \dots, b_k\} \cap \conv \{c_1, \dots, c_\ell\} \neq \emptyset.\]
This means that condition $(\ast)$ on the function $\varphi$ is witnessed by a finite set of points from each convex set in $F$. 

Therefore, if we replace each set in $F$ be the convex hull of its corresponding witness points, then we obtain a family of convex polytopes that satisfies the hypothesis of the theorem. Since each convex polytope is contained in its original member of $F$, a hyperplane transversal to a subfamily of the polytopes is also a hyperplane transversal to the corresponding subfamily of original members of $F$.

\medskip

The next step of the proof is to linearize the problem. This means that we consider $F$ as a finite family of convex polytopes living in an affine hyperplane in $\mathbb{R}^{d+1}\setminus\{0\}$. Similarly, we consider the image of the map $\varphi$ to live in an affine hyperplane in $\mathbb{R}^{k+1}\setminus \{0\}$. We define $\check{F} = F\cup (-F)$, where $-F = \{-P : P\in F\}$, and note that reflection through the origin is a  fixed-point-free involution on $\check{F}$.

The rank function $\mu$ is extended to a matroid on $\check{F}$ as follows: 
For a subfamily $G\subset \check{F}$, let $\tilde{G} = \{C\in F : C\in G \text{ or } -C\in G \}$, and define \[\mu(G) = \mu(\tilde{G}).\]

We leave it to the reader  to verify that this extends $\mu$ to a  the rank function of a matroid on $\check{F}$. (This is just a sequence of parallel extensions of our original matroid on $F$.) Note also that $\mu(\check{F} \setminus G) \geq \mu( F\setminus \tilde{G})$ for any subfamily $G\subset \check{F}$.

The function $\varphi$ is extended to $\check{F}$ by defining $\varphi(-P) = -\varphi(P)$ for every $P\in F$. This gives us an equivariant map  \[\varphi : \check{F} \to \mathbb{R}^{k+1}\setminus\{0\}.\]

We also need to linearize the 
hypothesis $(\ast)$ on the map $\varphi$. 
Observe that if $A$ and $B$ are compact convex sets contained in an affine hyperplane in $\mathbb{R}^n\setminus \{0
\}$, then 
\[A\cap B = \emptyset \iff 0 \notin \conv\big(A\cup (-B) \big).\]
Therefore, our equivariant map $\varphi: \check{F} \to \mathbb{R}^{k+1}$ satisfies  the condition
\[    0\notin \conv\: G \implies 0\notin \conv \: \varphi(G) \tag{$\check{\ast}$}\]
whenever $\mu(G) = |G|$.

\medskip

Finally, note that a subfamily $G\subset F$ has a $(d-1)$-transversal (an affine hyperplane transversal in $\mathbb{R}^d$) if and only if the subfamily $G \cup (-G) \subset \check{F}$ has a  hyperplane transversal through the origin in $\mathbb{R}^{d+1}$. 
The conclusion of the theorem can therefore be restated as:

\begin{quote}
{\em
There exists a subfamily $G\subset \check{F}$, with $\mu(\check{F}\setminus G)\leq k+1$, which has a hyperplane transversal through the origin in $\mathbb{R}^{d+1}$.}
\end{quote}

\medskip

Now that the problem has been linearized,  we can proceed with the proof of the Theorem. Define an abstract simplicial complex $K$ with vertex set $\check{F}$ by setting
\[K = \{G \subset \check{F} : {\mu}(G) = |G|, 0\notin \conv\: G \}.\]

Note that reflection through the origin in $\mathbb{R}^{d+1}$ induces a free $\mathbb{Z}_2$-action on $K$, which makes $K$ a free $\mathbb{Z}_2$-complex. (By a slight abuse of notation, we use ``$-$'' to denote the $\mathbb{Z}_2$-action on  $\mathbb{R}^{n}$ as well as on $K$.)

\begin{lemma*} Suppose ${\mu}(\check{F}\setminus G) > k+1$ for any subfamily $G\subset \check{F}$ that has a hyperplane transversal through the origin.
Then there exists a continuous $\mathbb{Z}_2$-map from $S^{k+1}$ to $||K||$.
\end{lemma*}

\begin{proof}
Let $V$ denote the union of all the vertices of the polytopes in $\check{F}$. 
The collection of orthogonal complements $\{v^\perp\}_{v\in V}$ forms an arrangement of hyperplanes through the origin in $\mathbb{R}^{d+1}$, and its intersection with $S^d$ induces an antipodally symmetric regular cell decomposition of $S^d$ which we denote by $\cal C$. 

Define $L$ to be the $(k+1)$-skeleton of the barycentric subdivision of $\cal C$. This means that $L$ is a simplicial complex whose vertices  are the cells of $\cal C$, and whose maximal simplices are ordered chains of cells in $\cal C$
\[\sigma_0 \subset \sigma_1 \subset \cdots \subset \sigma_{k+1}\]
where $\dim \sigma_i < \dim \sigma_{i+1}$. (We also use ``$-$'' to denote the antipodailty action on $\cal C$ and $L$.) 
Since $L$ is the $(k+1)$-skeleton of an antipodally symmetric triangulation of $S^d$, it follows that $L$ is a $k$-connected $\mathbb{Z}_2$-complex.
Consequently there exists a continuous $\mathbb{Z}_2$-map from $S^{k+1}$ to $||L||$ (see e.g. \cite[Proposition 5.3.2(iv)]{bubook}), and so it suffices to construct a continuous $\mathbb{Z}_2$-map $f: ||L|| \to ||K||$. We define $f$ inductively from the $i$-skeleton of $L$ to $K$, for $i=0, 1, \dots, k+1$. 

For each cell $\sigma\in \cal C$ (or in other words, for each vertex of $L$) we associate a subfamily $\check{F}(\sigma)\subset \check{F}$ in the following way: Choose an arbitrary point $y\in \text{relint}(\sigma)$ and define 
\[\check{F}(\sigma) = \{P \in \check{F} : x\cdot y >0 \: \forall x\in P\}.\]
Note that $\check{F}(\sigma)$ does not depend on the choice of $y$, since a change in $\check{F}(\sigma)$ can only  occur when a member of $\check{F}$ becomes (or stops being) tangent to the hyperplane $x\cdot y = 0$, as $y$ varies continuously.  Equivalently, a change in $\check{F}(\sigma)$ can only occur when $y$ enters (or leaves) the orthogonal complement of one of the vertices of a member of $\check{F}$. This shows that $\check{F}(\sigma)$ is well-defined. Note also that  $\check{F}(-\sigma) = - \check{F}(\sigma) = \{-P : P\in \check{F}(\sigma)\}$ for any cell $\sigma\in \cal C$.

By the same reasoning as above, we see that for any ordered chain of cells in $\cal C$
\[\sigma_1 \subset \sigma_2 \subset \cdots \subset \sigma_{i} \] the associated subfamilies satisfy the following {\em monotonicity property}:
\[\check{F}(\sigma_1) \subseteq \check{F}(\sigma_2) \subseteq \cdots \subseteq \check{F}(\sigma_{i}).\]

Indeed, going from $\sigma_j$ to $\sigma_{j-1}$ means that we are adding additional constraints  in the form of additional vertices of polytopes being contained in the hyperplane $x\cdot y = 0$. Therefore we must have  $\check{F}(\sigma_{j-1})\subseteq \check{F}(\sigma_{j})$.

\smallskip

Next, we claim that for any cell $\sigma \in \cal C$, the induced subcomplex $K[\check{F}(\sigma)]$ is $k$-connected. Indeed, the subfamily $\check{F} \setminus \big(\check{F}(\sigma) \cup \check{F}(-\sigma)\big)$ has a hyperplane transversal through the origin, so by the hypothesis we have 
\[\mu\big(\check{F}(\sigma)\big) = \mu\big(\check{F}(\sigma) \cup \check{F}(-\sigma)\big)> k+1.\]
Thus $K[\check{F}(\sigma)]$ is a matroid complex of dimension $\geq k+1$, and it follows from \cite[Proposition 6.2]{BKL1985} that  $K[\check{F}(\sigma)]$ is $k$-connected. 

\medskip

We are ready to define the map $f$, starting from the $0$-skeleton of $L$. For each cell $\sigma \in \cal C$ (vertex of $L$) choose an arbitrary point $f(\sigma)$ in the induced subcomplex $K[\check{F}(\sigma)]$ such that it commutes with the antipodality action, that is, such that $f(-\sigma) = -f(\sigma)$. For the induction step we suppose that $f$ has been defined on the $(i-1)$-skeleton of $L$, for some $1\leq i \leq k+1$, such that 
\begin{itemize}
    \item $f$ commutes with the antipodality action, and
    \item the image of an $(i-1)$-simplex $\sigma_1 \subset \sigma_2 \subset \cdots \subset \sigma_i$ is contained in the induced subcomplex $K[\check{F}(\sigma_i)]$.
\end{itemize}

We will extend $f$ to the $i$-skeleton of $L$ using the $k$-connectedness of the induced subcomplexes $K[\check{F}(\sigma)]$. Consider an $i$-dimensional simplex of $L$
\[\sigma_1 \subset \sigma_2 \subset \cdots \subset \sigma_{i+1}.\] 
By the induction hypothesis and the monotonicity property, the boundary of this simplex is mapped by $f$ into the induced subcomplex $K[\check{F}(\sigma_{i+1})]$, which is $k$-connected. Therefore the map can be extended to a map from the entire simplex into $K[\check{F}(\sigma_{i+1})]$. Once this extension has been chosen, we extend the map on the andtipodal simplex 
\[-\tau_1 \subset -\tau_2 \subset \cdots \subset -\tau_{i+1}\]
so that it commutes with the $\mathbb{Z}_2$-action. The Lemma now follows by induction. 
\end{proof}

To complete the proof of the Theorem, we claim there is a continuous $\mathbb{Z}_2$-map from $||K||$ to $\mathbb{R}^{k+1}\setminus\{0\}$.
Indeed, the equivariant map $\varphi: \check{F} \to \mathbb{R}^{k+1}\setminus\{0\}$ can be viewed as a map from the vertices of $K$ to $\mathbb{R}^{k+1}\setminus\{0\}$, and 
condition $(\check{\ast})$ states that $0\notin \conv\: \varphi(G)$ whenever $G\subset \check{F}$ is a simplex of $K$. Thus the affine extension of $\varphi$ gives the desired map.

By composing the affine extension of $\varphi$ and the map from the Lemma  we get a continuous $\mathbb{Z}_2$-map from $S^{k+1}$ to $\mathbb{R}^{k+1}\setminus \{0\}$, which contradicts the Borsuk--Ulam theorem. \qed

\section{Acknowledgements}
Thank you to Minho Cho for corrections on an earlier version of this manuscript.

\end{document}